\documentclass{article}

\usepackage{amsmath, amsthm, amsfonts, amssymb}
\usepackage[abbrev]{amsrefs}
\usepackage{hyperref}
\usepackage{authblk}

\newtheorem{thm}{Theorem}[section]

\newtheorem{lem}[thm]{Lemma}
\newtheorem{prop}[thm]{Proposition}
\newtheorem*{claim}{Claim}
\theoremstyle{definition}
\newtheorem{defn}[thm]{Definition}
\theoremstyle{remark}
\newtheorem{rem}[thm]{Remark}
\newtheorem*{ack}{Acknowledgements}

\def\QQ{\mathbb{Q}}
\def\CC{\mathbb{C}}
\def\AAA{\mathbb{A}}
\def\PP{\mathbb{P}}
\def\SL{\mathrm{SL}}

\title{Crepant resolution of $\AAA^4/A_4$ in characteristic $2$}
\author{Linghu Fan}
\affil{Graduate School of Mathematical Sciences, the University of Tokyo\\E-mail: linghu.fan@ipmu.jp}
\date{}

\begin{document}
\maketitle

\abstract{In this paper, we construct a crepant resolution for the quotient singularity $\AAA^4/A_4$ in characteristic $2$, where $A_4$ is the alternating group of degree $4$ with permutation action on $\AAA^4$. By computing the Euler number of the crepant resolution, we obtain a new counterexample to an analogous statement of McKay correspondence in positive characteristic.}

\section{Introduction}

Let $K$ be an algebraically closed field and $X$ be an algebraic normal variety over $K$. For a resolution $f: Y \to X$, $f$ is called crepant if $K_Y=f^*K_X$. Our interest in crepant resolutions comes from McKay correspondence. As a generalized version of McKay correspondence over $\CC$, Batyrev's theorem tells that quotient singularities with crepant resolutions have a fine property:
\begin{thm}[\cite{batyrev1999non}, Theorem 1.10.]
Let $G$ be a finite subgroup of $\SL (n,\CC)$ acting on $\CC^n$. Assume that there exists a crepant resolution $f: Y\to \CC^n/G$. Then the Euler number of $Y$ is equal to the number of conjugacy classes of $G$.
\end{thm}

In dimension $2$, minimal resolutions of quotient singularities $\CC^2/G$ are crepant, and Batyrev's theorem becomes a corollary of classical McKay correspondence. In dimension $3$, crepant resolution for any possible $\CC^3/G$ exists, according to constructions by Markushevich \cite{markushevich1987description}, \cite{markushevich1997resolution}, Roan \cite{roan1989generalization}, \cite{roan1994c1}, \cite{roan1996minimal} and Ito \cite{ito1994crepant}, \cite{ito1995gorenstein}. For higher dimensions, there are examples of quotient singularities with no crepant resolutions. 

We consider the analogous statement of Batyrev's theorem in positive characteristic, where the field is $K$, an algebraically closed field of characteristic $p>0$, instead of $\CC$. To determine Euler number in positive characteristic, we use the following definition.
\begin{defn}
Fix a prime $l \neq p$. Denote the $l$-adic cohomology with compact support by $H^{i}_c(-,\QQ_l)$. Let $X$ be a smooth algebraic variety over $K$. Then we define Euler number of $X$ to be
$$\chi(X)=\sum_i(-1)^{i}\mathrm{dim}_{\QQ_l}H^{i}_c(X,\QQ_l).$$
\end{defn}
Note that this definition is independent of choice of $l$, and it coincides with the definition of topological Euler number in characteristic $0$.

In positive characteristic, for a finite subgroup $G \subseteq \SL(n,K)$, there are two cases: non-modular case, when $p$ does not divide the order of $G$; and modular case, when $p$ divides the order of $G$. Roughly speaking, non-modular cases are easier to be considered, since the associated quotient singularities can be lifted to $\CC$. In particular, Batyrev's theorem holds for non-modular quotient singularities in positive characteristic.  

For modular cases, few examples of crepant resolutions are known. Chen, Du and Gao \cite{chen2020modular} gave a crepant resolution as a counterexample to Batyrev's theorem in characteristic $2$. In their example, the group $G\cong C_6 \subseteq \SL(2,K)$ has a reflection. Yasuda \cite{yasuda2014cyclic} showed that Batyrev's theorem holds for the cases when the group $G$ is $p$-cyclic with no reflections, and gave two examples with crepant resolutions: $\AAA^4_K/C_2$ ($p=2$) and $\AAA^3_K/C_3$ ($p=3$). For groups with more complicated structure, even if we assume that the group has no reflections, there is still a counterexample given by Yamamoto \cite{yamamoto2021crepant}: in characteristic $3$, the quotient singularity $\AAA^3_K/S_3$ has a crepant resolution with Euler number $6$, while the symmetric group $S_3$ has $3$ conjugacy classes. 

In the known examples above, the Sylow $p$-subgroups are $p$-cyclic. In this paper, we construct a crepant resolution of quotient singularity in characteristic $2$, where the group $G$ has no reflections, and has a non-cyclic Sylow 2-subgroup of order $2^2$:

\begin{thm}[Main result]
Let $K$ be an algebraically closed field of characteristic $2$, and $A_4$ be the alternating group with permutation action on $\AAA^4_K$. Denote the quotient singularity $\AAA^4_K/A_4$ by $X$. Then $X$ has a crepant resolution $\widetilde{X}$ with Euler number $\chi (\widetilde{X})=10$.
\end{thm}

Since the alternating group $A_4$ has $4$ conjugacy classes, our result is also a new counterexample to analogous statement of Batyrev's theorem in positive characteristic. 

\begin{ack}
This paper contains some results in the author's master thesis. The author is grateful to his advisor, Professor Yukari Ito, for her useful suggestions. The author would also like to thank Professor Takehiko Yasuda for his valuable advice. The author is supported by WINGS-FMSP at the University of Tokyo.
\end{ack}

\section{Preliminaries}
To give a proof of the main result, we firstly list propositions to study the given quotient singularity $\AAA^4/A_4$ and crepant morphisms, especially in characteristic $2$.   
\begin{prop}\label{A4equation}
Let $K$ be an algebraically closed field of characteristic 2. Under the permutation action of $A_4$,
\begin{align*}
   \mathbb{A}^4_K /A_4 \cong V(&E^2+(A^2D+ABC+C^2)E+A^4D^2+A^3C^3+A^2B^3D\\&+B^3C^2+C^4) . 
\end{align*}

\end{prop}

\begin{proof}

It is known that in general, under the permutation action of $A_4$, the invariant ring $K[x_1,x_2,x_3,x_4]^{A_4}=K[s_1,s_2,s_3,s_4,\Delta_4]$, where $s_i$ are the elementary symmetric polynomials and $\Delta_4$ is an $A_4$-invariant polynomial of degree 6 which is not symmetric. Over fields of characteristic different from 2, $\Delta_4$ is always taken as the Vandermonde polynomial. However, in characteristic 2, one should take $\Delta_4=\mathcal{O}_{A_4}(x_1^3x_2^2x_3)$ instead, where $\mathcal{O}_G$ denotes the orbit sum under the group action (see \cite{campbell2011modular}, 4.4). 

In characteristic 2, computation shows that $\Delta_4^2+(s_1^2s_4+s_1s_2s_3+s_3^2)\Delta_4$ is a symmetric polynomial of degree 12 with elementary representation $s_1^4s_4^2+s_1^3s_3^3+s_1^2s_2^3s_4+s_2^3s_3^2+s_3^4$. 
Let $R=K[A,B,C,D,E]$ be a graded polynomial ring such that $A,B,C,D,E$ are of degree $1,2,3,4,6$ respectively. Consider the ring homomorphism $\phi:R \to K[x_1,x_2,x_3,x_4]^{A_4}$ determined by $\phi(A)=s_1, \phi(B)=s_2, \phi(C)=s_3, \phi(D)=s_4, \phi(E)=\Delta_4$. Then $\phi$ is a surjective homomorphism between two graded rings, and $f:=E^2+(A^2D+ABC+C^2)E+A^4D^2+A^3C^3+A^2B^3D+B^3C^2+C^4 \in \mathrm{Ker}\phi$. 

We want to show that $R/(f)=R/\mathrm{Ker}\phi \cong K[x_1,x_2,x_3,x_4]^{A_4}$. For a graded ring $S$, we consider its Hilbert series $\mathcal{H}(S,\lambda) = \sum_{d \geq 0}\lambda ^{d}\dim_{K}S_{d}$. Since $R$ is generated by elements of degree $1,2,3,4,6$ and $f$ is of degree $12$, we have
\begin{align*}
 \mathcal{H}(R/(f),\lambda)
=\frac{1-\lambda^{12}}{(1-\lambda)(1-\lambda^2)(1-\lambda^3)(1-\lambda^4)(1-\lambda^6)}.   
\end{align*}
On the other hand, since $A_4$ has a permutation action on $\AAA^4$, the Hilbert series of invariant ring does not change by characteristic of the field, and it is possible to compute the Hilbert series by Molien's theorem over $\CC$ (see \cite{campbell2011modular}, 3.7 and 4.5). Therefore,
\begin{align*}
&\mathcal{H}(K[x_1,x_2,x_3,x_4]^{A_4},\lambda)\\
=&\frac{1}{12}(\frac{1}{(1-\lambda)^4}+\frac{3}{(1-\lambda)^2(1+\lambda)^2}+\frac{8}{(1-\lambda)(1-\lambda^3)})\\
=&\frac{1+\lambda^6}{(1-\lambda)(1-\lambda^2)(1-\lambda^3)(1-\lambda^4)}\\
=&\frac{1-\lambda^{12}}{(1-\lambda)(1-\lambda^2)(1-\lambda^3)(1-\lambda^4)(1-\lambda^6)}.
\end{align*}
Since $R/(f)$ and $K[x_1,x_2,x_3,x_4]^{A_4}$ share the same Hilbert series, the induced surjective homomorphism $\widetilde{\phi}:R/(f) \to K[x_1,x_2,x_3,x_4]^{A_4}$, which is surjective on each degree between two $K$-linear spaces with the same dimension, should also be isomorphic on each degree. Therefore $R/(f) \cong K[x_1,x_2,x_3,x_4]^{A_4}$. 
\end{proof}
By Proposition \ref{A4equation}, the quotient singularity $\AAA^4/A_4$ is isomorphic to a hypersurface with a computable defining equation. Using this equation, we will show that crepant resolution of the singularity can be obtained by composition of a series of blow-ups.

\begin{lem}\label{crepantcondition}
Let $X$ be a hypersurface in $\AAA^n$. Consider the blow-up of $X$ along $C \subseteq X$. If $C$ is smooth and irreducible, with codimension $3$ in $\AAA^n$, and $X$ has multiplicity $2$ along $C$, then the blow-up morphism $f:\widetilde{X}\to X$ is crepant.
\end{lem}
\begin{proof}
    Denote the blow-up of $\AAA^n$ along $C$ again by $f:U\to \AAA^n$. Then by abusing notations of exceptional divisors, we have
    \begin{align*}
        K_{U}&=f^*K_{\AAA^n}+2E,\\
        f^*X&=\widetilde{X}+2E.
    \end{align*}
    Taking them together, and applying the adjunction formula, we obtain
    \begin{align*}
        K_{\widetilde{X}}=(K_{U}+\widetilde{X})|_{\widetilde{X}}=(f^*K_{\AAA^n}+f^*X)|_{\widetilde{X}}=f^*(K_{\AAA^n}+X)|_{\widetilde{X}}=f^*K_X.
    \end{align*}
    Therefore, $f$ is crepant.
\end{proof}
Lemma \ref{crepantcondition} is a sufficient condition for a blow-up morphism to be crepant. We will see that all the blow-ups in the paper meet this condition.

\begin{prop}\label{conics}
Let $C=(h=0)$ be a conic in $\AAA^2_K$, where $K$ is an algebraically closed field of characteristic $2$, and $h(x,y)=ax^2+bxy+cy^2+dx+ey+f$. Here $a,b,c$ are not all $0$. Then $C$ is determined as follows.
\begin{enumerate}
    \item $C$ is a double line, if $b=d=e=0$.
    \item $C$ is two intersecting different lines, if one of the following holds:
    \begin{itemize}
        \item $b \neq 0, f+\frac{1}{b}(de+\frac{ae^2}{b}+\frac{cd^2}{b})=0$; 
        \item $a,d \neq 0, b=e+d\sqrt{\frac{c}{a}}=0$;
        \item $c,e \neq 0, b= d+e\sqrt{\frac{a}{c}}=0$.
    \end{itemize}    
    \item $C$ is a non-degenerate conic, otherwise.
\end{enumerate}
\end{prop}
\begin{proof}
If $b \neq 0$, then we can take $\alpha \in K$ such that $a\alpha ^2+b\alpha +c = 0$. Therefore
$$
h(x,y)=bXY+F,
$$
where $X=x+\alpha y+\frac{e+d\alpha}{b}$, $Y=y+\frac{a}{b}(x+\alpha y)+\frac{1}{b}(d+\frac{a}{b}(e+d\alpha))$, $F=f+\frac{1}{b}(e+d\alpha)(d+\frac{a}{b}(e+d\alpha))=f+\frac{1}{b}(de+\frac{ae^2}{b}+\frac{cd^2}{b})$. Hence $C$ is non-degenerate if and only if $F \neq 0$; when $F=0$, $C$ becomes two intersecting lines.

If $b=0$, we may assume that $a\neq 0$ without loss of generality. Then
$$
h(x,y)=aX^2+dX+Ey+f,
$$
where $X=x+\sqrt{\frac{c}{a}}y$, $E=e+d\sqrt{\frac{c}{a}}$. Here if $E=0$, then $C$ degenerates as two lines; in particular, the two lines become a double line if $d$ is furthermore $0$. If $E \neq 0$, then
$$
h(x,y)=aX_1^2+EY,
$$
where $X_1=X+\sqrt{\frac{f}{a}}$, and $Y=y+\frac{d}{E}X$. Thus $C$ is non-degenerate under this assumption.
\end{proof}
\begin{rem}\label{conicpj}
The proposition above can also be applied to determine the classification of conics in $\mathbb{P}^2_K$ with defining equation $h(X,Y,Z)=aX^2+bXY+cY^2+dXZ+eYZ+fZ^2$, except when $a=b=c=0$. When $h(X,Y,Z)=dXZ+eYZ+fZ^2$, it is obvious that the conic is degenerate as two projective lines, and the two lines become one only if $d=e=0$. 
\end{rem}
By Proposition \ref{conics} and Remark \ref{conicpj}, we can check that a projective conic in characteristic $2$ is isomorphic to $\PP^1$ (a double line or a non-degenerate conic) or $\PP^1 \vee \PP^1$ (two different lines), and that will be helpful when we compute Euler number of the resolution by considering the exceptional divisors.

\section{Proof of the main result}
By Proposition \ref{A4equation}, in characteristic 2, quotient singularity $\mathbb{A}_{K}^{4}/A_4$ by permutation is isomorphic to the 4-dimensional hypersurface $M=V(f)\subseteq \AAA^5_K$, where
\begin{align*}
  f = E^2+(A^2D+ABC+C^2)E+A^4D^2+A^3C^3+A^2B^3D+B^3C^2+C^4.  
\end{align*}
It is possible to obtain $\mathrm{Sing}(M)$ by direct computation using the defining equation, but here we choose to use another way to compute it, in order to apply the idea that comes from the construction of crepant resolution of $\mathbb{C}^3/H_{168}$ by Markushevich \cite{markushevich1997resolution}, where $H_{168}$ is a simple subgroup of $\mathrm{SL}(3,\mathbb{C})$ of order 168. 

By considering the permutation action of $A_4$ on $\mathbb{A}^4_K$, there are 3 planes fixed by elements of order 2, and 4 planes fixed by elements of order 3. By the quotient map, these 2 families of fixed planes give the singular locus of quotient variety as union of their images. To obtain the defining equation of singular locus, it suffices to consider the image of representative plane from each family. For the family of planes fixed by an element of order 2, we can take $\{x_1=x_2,x_3=x_4\}$ as a representative. Then by checking its parametrised form $(x_1,x_2,x_3,x_4)=(t_1,t_1,t_2,t_2)$ and using the formula for $(A,B,C,D,E)$ in the construction of hypersurface, we obtain parametrisation of one singular plane in the quotient variety:
\begin{align*}
 P_1=\{A=0,B=t_1^2+t_2^2,C=0,D=t_1^2t_2^2,E=0\}
=V(A,C,E).   
\end{align*}
By similar procedure for the plane $\{x_1=x_2=x_3\}$ fixed by an element of order 3, the other singular plane is written as:
\begin{align*}
P_2 =& 
\{A=t_1+t_2,B=t_1^2+t_1t_2,C=t_1^3+t_1^2t_2,D=t_1^3t_2,\\&E=t_1^6+t_1^5t_2+t_1^4t_2^2+t_1^3t_2^3\}\\
=&V(B^2+AC,ABC+A^2D+C^2,E+A^2D+C^2).  
\end{align*}
Note that to show the second equality, one should prove double inclusion between two sets. It is easy to see that $P_2$ is contained in the variety given by 3 defining equations. Conversely, given a point in $V(B^2+AC,ABC+A^2D+C^2,
E+A^2D+C^2)$, one can take $t_1=\frac{B}{A}$ ($A\neq 0$) or $\sqrt[4]{D}$ ($A=0$), and $t_2=A+t_1$, to certify that the point is exactly in $P_2$ parametrised by $(t_1,t_2)$.

Here we obtain $\mathrm{Sing}(M) = P_1 \cup P_2$. According to Markushevich's construction in characteristic 0, one may hope that the singularity would be resolved via a series of blow-ups as following: we first compute the blow-up of $M$ along $P_1$, and then repeatedly compute the blow-up along the singular part of exceptional divisor of previous blow-up, until the whole singular locus becomes exactly the strict transform of $P_2$. Then $\widetilde{P_2}$ has singularities from the action of non-modular elements of order $3$, hence 
the final blow-up along the whole singular locus can give the resolution. The following claim tells that this procedure does give a crepant resolution of the quotient singularity. 

\begin{claim}
Let $\pi_1:U \to M$ be the blow-up of $M$ along $P_1$ with exceptional divisor $E_1$, $\pi_2:V \to U$ be the blow-up of $U$ along $E_1 \cap \mathrm{Sing}(U)$ with exceptional divisor $E_2$, $\pi_3:W \to V$ be the blow-up of $V$ along $E_2 \cap \mathrm{Sing}(V)$ with exceptional divisor $E_3$, and $\pi_4:R \to W$ be the blow-up of $W$ along $\mathrm{Sing}(W)$ with exceptional divisor $E_4$. Then $\pi:=\pi_1 \circ \pi_2 \circ \pi_3 \circ \pi_4:R \to M$ is a crepant resolution of $M$. 
\end{claim}
\begin{proof}
To show that $\pi$ is a crepant resolution, it suffices to check that each blow-up is along a smooth locus of codimension 3 in the whole space, and that the hypersurface has multiplicity $2$ along the center of blow-up (such that each blow-up is a crepant morphism by Lemma \ref{crepantcondition}), and that $R$ is smooth. What is more, for each blow-up, we can use Proposition \ref{conics} to check the structure of exceptional divisor.

\textbf{Step 1:} $\pi_1:U \to M$.
Take projective coordinates $(u_0:u_1:u_2)=(A:C:E)$. Then
\begin{align*}
U=V(&u_2^2+(Du_0^2+Bu_0u_1+u_1^2)E+A^2D^2u_0^2+AC^3u_0^2+B^3Du_0^2\\&+B^3u_1^2+C^2u_1^2, (u_0:u_1:u_2)=(A:C:E)).
\end{align*}
Here $E_1=V(u_2^2+B^3Du_0^2+B^3u_1^2).$ By the base change given by the Frobenius $\mathrm{Spec}(K[\sqrt{B},\sqrt{D}])\to \mathrm{Spec}(K[B,D])$, we have $u_2^2+B^3Du_0^2+B^3u_1^2=(u_2+\sqrt{B^3D}u_0+\sqrt{B^3}u_1)^2$, such that $E_1$ can be viewed as a trivial $\mathbb{P}^1$-bundle over $\mathbb{A}^2$ after the base change. Since the Frobenius is a universal homeomorphism, this base change does not change Euler numbers. By abuse of notation, we write $E_1\cong \mathbb{A}^2 \times \mathbb{P}^1$ under the necessary base change. In the defining equation of $U$, $u_0=u_1=0$ implies $u_2=0$, which gives a contradiction, so $U$ is covered by its affine pieces determined by $u_0 \neq 0$ and $u_1 \neq 0$ respectively. We denote them by $U_0=U\cap \{u_0 \neq 0\}$ and $U_1=U\cap \{u_1 \neq 0\}$. For these 2 affine pieces, we have
\begin{align*}
U_0 \cong V(&u_2^2+(D+Bu_1+u_1^2)Au_2+A^2D^2+A^4u_1^3+B^3D+B^3u_1^2+A^2u_1^4),\\
\mathrm{Sing}(U_0)&\cap E_1 =V(A,B,u_2).\\
U_1 \cong V(&u_2^2+(Du_0^2+Bu_0+1)Cu_2+C^2D^2u_0^4+C^4u_0^3+B^3Du_0^2+B^3+C^2),\\
\mathrm{Sing}(U_1)&\cap E_1 =V(C,B,u_2).
\end{align*}
And we also obtain $E_1 \setminus \mathrm{Sing}(U) \cong \mathbb{A}^2 \times \mathbb{P}^1 \setminus \mathbb{A}^1 \times \mathbb{P}^1$ by gluing its affine pieces together.

\textbf{Step 2:} $\pi_2:V \to U$.
From step 1, it suffices to consider blow-ups of $U_0$ along $V(A,B,u_2)$ and $U_1$ along $V(C,B,u_2)$, and then glue them together to obtain $V$. Denote the 2 blow-ups by $\widetilde{U_0}$ and $\widetilde{U_1}$. Then
\begin{align*}
\widetilde{U_0} = V(&v_2^2+(D+Bu_1+u_1^2)v_0v_2+D^2v_0^2+A^2u_1^3v_0^2+BDv_1^2+Bu_1^2v_1^2\\&+u_1^4v_0^2,(v_0:v_1:v_2)=(A:B:u_2)),\\
\widetilde{U_1} = V(&v_2^2+(Du_0^2+Bu_0+1)v_0v_2+D^2u_0^4v_0^2+C^2u_0^3v_0^2+BDu_0^2v_1^2+Bv_1^2\\&+v_0^2,(v_0:v_1:v_2)=(C:B:u_2)).
\end{align*}
Here $E_2=V(v_2^2+(Du_0^2+u_1^2)v_0v_2+(Du_0^2+u_1^2)^2v_0^2)$ gives two intersecting projective lines if $Du_0^2+u_1^2 \neq 0$ or a double line $\{v_2^2=0\}$ if $Du_0^2+u_1^2 = 0$. Similarly to what is done in step 1, we may write 
$$
E_2 \cong (\mathbb{A}^1 \times \mathbb{P}^1 \setminus \mathbb{A}^1)\times (\mathbb{P}^1 \vee \mathbb{P}^1) \cup \mathbb{A}^1 \times \mathbb{P}^1.
$$
For $\widetilde{U_1}$, computation shows that $\mathrm{Sing}(\widetilde{U_1})\cap E_2 \subseteq V\cap \{u_0 \neq 0\} \subseteq \widetilde{U_0}$, thus we only need to consider $\widetilde{U_0}$. Denote by $V_0$ and $V_1$ the affine pieces of $\widetilde{U_0}$ determined by $v_0 \neq 0$ and $v_1 \neq 0$ respectively, and then $\widetilde{U_0} = V_0 \cup V_1$. Again like what is done in step 1, we obtain
\begin{align*}
V_0 \cong V(&v_2^2+(D+Au_1v_1+u_1^2)v_2+D^2+A^2u_1^3+ADv_1^3+Au_1^2v_1^3+u_1^4),\\
\mathrm{Sing}(V_0)&\cap E_2 =V(A,v_2,D+u_1^2).\\
V_1 \cong V(&v_2^2+(D+Bu_1+u_1^2)v_0v_2+D^2v_0^2+B^2u_1^3v_0^4+BD+Bu_1^2+u_1^4v_0^2),\\
\mathrm{Sing}(V_1)&\cap E_2 =V(B,v_2,D+u_1^2).
\end{align*}
It follows that the part $\mathbb{A}^1 \times \mathbb{P}^1$ in $E_2$ is exactly singular part, thus $E_2 \setminus \mathrm{Sing}(V)\cong (\mathbb{A}^1 \times \mathbb{P}^1 \setminus \mathbb{A}^1)\times (\mathbb{P}^1 \vee \mathbb{P}^1).$

\textbf{Step 3:} $\pi_3:W \to V$.
With similar notations as above, taking $(w_0:w_1:w_2)$ as projective coordinates for $(A:v_2:D+u_1^2)$ in $\widetilde{V_0}$ and $(B:v_2:D+u_1^2)$ in $\widetilde{V_1}$, we have
\begin{align*}
\widetilde{V_0} = V(&w_1^2+(w_2+u_1v_1w_0)w_1+w_2^2+v_1^3w_0w_2+u_1^3w_0^2,\\&(w_0:w_1:w_2)=(A:v_2:D+u_1^2)),\\
\widetilde{V_1} = V(&w_1^2+(w_2+u_1w_0)v_0w_1+v_0^2w_2^2+u_1^3v_0^4w_0^2+w_0w_2,\\&(w_0:w_1:w_2)=(B:v_2:D+u_1^2)).
\end{align*}
Consider the conics with coordinates $(w_0:w_1:w_2)$ in $E_3$. Computation shows that they are degenerate as 2 different projective lines exactly when $v_0 \neq 0$ and $u_1v_0^2=v_1^2$. Therefore,
$$
E_3 \cong (\mathbb{A}^1 \times \mathbb{P}^1 \setminus \mathbb{A}^1)\times \mathbb{P}^1  \cup \mathbb{A}^1 \times (\mathbb{P}^1\vee \mathbb{P}^1).
$$
For $\widetilde{V_1}$, again by computation we obtain $\mathrm{Sing}(\widetilde{V_1})\cap E_3 \subseteq W\cap \{v_0 \neq 0\} \subseteq \widetilde{V_0}$, so it suffices to consider $\widetilde{V_0}$. Using similar notations as step 1 and step 2, we consider the affine cover $\widetilde{V_0}=W_0 \cup W_2$ as follows.
\begin{align*}
&W_0 \cong V(w_1^2+w_1w_2+u_1v_1w_1+w_2^2+v_1^3w_2+u_1^3),\\
&\mathrm{Sing}(W_0) =\{w_1=w_2=v_1^3,u_1=v_1^2\}.\\
&W_2 \cong V(w_1^2+(1+u_1v_1w_0)w_1+1+v_1^3w_0+u_1^3w_0^2),\\
&\mathrm{Sing}(W_2)\subseteq \mathrm{Sing}(W)\cap \{w_0\neq 0\} \subseteq \mathrm{Sing}(W_0).
\end{align*}
Therefore $\mathrm{Sing}(W)=\mathrm{Sing}(W_0)$ is exactly the strict transform of $P_2$ under these 3 blow-ups above, and we only need to consider $W_0$ for the next step. As for $E_3 \setminus \mathrm{Sing}(W)$, since singularities in $E_3$ only appear at each crossing point of projective lines, we obtain
\begin{align*}
&E_3 \setminus \mathrm{Sing}(W) \\
\cong &(\mathbb{A}^1 \times \mathbb{P}^1 \setminus \mathbb{A}^1)\times \mathbb{P}^1  \cup \mathbb{A}^1  \times (\mathbb{P}^1\vee \mathbb{P}^1 \setminus \{1 point\}).
\end{align*}

\textbf{Step 4:} $\pi_4:R \to W$.
Let $w_1'=w_1+v_1^3,w_2'=w_2+v_1^3,u_1'=u_1+v_1^2$, and use the projective coordinates $(r_0:r_1:r_2)=(w_1':w_2':u_1')$. Then
\begin{align*}
\widetilde{W_0}=V(&r_0^2+r_1^2+r_0r_1+u_1'r_2^2+v_1r_0r_2+v_1^2r_2^2,\\&(r_0:r_1:r_2)=(w_1':w_2':u_1')).
\end{align*}
For the exceptional divisor, $E_4=V(r_0^2+r_0(r_1+v_1r_2)+(r_1+v_1r_2)^2)\cong \mathbb{A}^2 \times (\mathbb{P}^1 \vee \mathbb{P}^1)$(note that $A$ is a hidden affine coordinate in the defining equation of $\widetilde{W_0}$). For the smoothness of $\widetilde{W_0}$, by considering its affine pieces $\widetilde{W_0}=R_0\cup R_2$, we obtain
\begin{align*}
&R_0 \cong V(1+r_1^2+r_1+w_1'r_2^3+v_1r_2+v_1^2r_2^2),\\
&R_2 \cong V(r_0^2+r_1^2+r_0r_1+u_1'+v_1r_0+v^2).
\end{align*}
And it is easy to check that both $R_0$ and $R_2$ are smooth.

Above all, we know that $R$ is smooth and $\pi$ is composed of crepant blow-up morphisms, hence $\pi:R\to M$ is a crepant resolution of $M$. 
\end{proof}
For the calculation of Euler number,
\begin{align*}
\chi(R)=&\chi(E_4)+\chi(E_3\setminus \mathrm{Sing}(W))\\
&+\chi(E_2\setminus \mathrm{Sing}(V))+\chi(E_1\setminus \mathrm{Sing}(U))\\
&+\chi(M\setminus \mathrm{Sing}(M))\\
=&3+(2+2)+3+0+0=10.
\end{align*}

Then we finish the proof.

\begin{rem}
In characteristic other than $2$, $\mathbb{A}_{K}^{4}/A_4$ has a crepant resolution of Euler number $4$, which equals the number of conjugacy classes of $A_4$. Therefore, an analogue of Batyrev's theorem in characteristic 0 holds for $\mathbb{A}_{K}^{4}/A_4$ in any odd characteristic. For characteristic other than $2,3$, they are non-modular cases. For characteristic $3$, note that $A_4$ is unique up to conjugation as a subgroup of $\SL(4,K)$, and then see \cite{yamamoto2021crepant}, Theorem 1.2.
\end{rem}
\begin{rem}
From the perspective of motivic integral, we have $[R]=\mathbb{L}^4+6\mathbb{L}^3+3\mathbb{L}^2$, which is totally different from characteristic 0 case, where $[\widetilde{\mathbb{C}^4/A_4}]=\mathbb{L}^4+3\mathbb{L}^3$, and the coefficients are explained as numbers of conjugacy classes of age $0$ or $1$ in $A_4$.
\end{rem}
\begin{rem}
Let $H$ be the normal Sylow 2-subgroup of $A_4$ containing all elements of order $2$. Then $A_4=H \rtimes C_3$, and the quotient singularity $\AAA^4/A_4 \cong (\AAA^4/H)/C_3$. In characteristic $0$, crepant resolution of $\CC^4/A_4$ can be obtained by first taking $Y$, a $C_3$-equivariant crepant resolution of $\CC^4/H$, and then taking a crepant resolution of $Y/C_3$. Therefore, $\CC^4/H$ is regarded as an easier singularity than $\CC^4/A_4$. However, in characteristic $2$, we cannot use similar approach to obtain crepant resolution of $\AAA^4/H$: that is, $\AAA^4/H$ is not easier than $\AAA^4/A_4$ in a sense. This is again a special phenomenon of quotient singularities in positive characteristic.
\end{rem}

% Bibliography
%-----------------------------------------------------------------
\bibliography{ref}
\bibliographystyle{amsplain}

\end{document}